\documentclass[a4paper,11pt]{article}

\usepackage[utf8]{inputenc}
\usepackage[english]{babel}
\usepackage{amsmath,amsfonts,amssymb,latexsym,amsthm}
\usepackage{hyperref}

\hypersetup{
colorlinks=true,
linkcolor=blue,
filecolor=magenta,      
urlcolor=cyan,
}

\usepackage{caption}
\usepackage{algorithm}
\usepackage{algpseudocode}

\newtheorem{theorem}{Theorem}[section]
\newtheorem{lemma}[theorem]{Lemma}
\newtheorem{mainlemma}[theorem]{Main Lemma}
\newtheorem{corollary}[theorem]{Corollary}

\theoremstyle{definition}
\newtheorem{definition}[theorem]{Definition}
\newtheorem{remark}[theorem]{Remark}
\newtheorem{example}[theorem]{Example}
\newtheorem{conjecture}[theorem]{Conjecture}

\usepackage[left=1in, right=1in, top=1in, bottom=1in]{geometry} 
\usepackage{graphicx} 
\usepackage{tikz}

\newcommand{\ra}{\hspace{-2pt}\mathbin{\rightarrow}}
\newcommand{\sq}{\mathbin{\square}}
\renewcommand{\P}{\mathbf{P}}
\newcommand{\bP}{\mathbf{P}}
\def\cU{\mathcal U}
\def\cV{\mathcal V}
\def\cW{\mathcal W}

\title{Percolation Inequalities and Decision Trees}
\author{Nikita Gladkov}
\date{}

\begin{document}

\maketitle

\begin{abstract}
The use of decision trees for percolation inequalities started with the celebrated O'Donnell--Saks--Schramm--Servedio (OSSS) inequality. We prove decision tree generalizations of the Harris--Kleitman (HK), van den Berg--Kesten (vdBK), and other inequalities. These inequalities are then applied to estimate the connection probabilities in Bernoulli bond percolation on general graphs.
\end{abstract}

\section{Introduction}

In percolation theory, key tools include the Harris--Kleitman (HK) and van den Berg--Kesten (vdBK) inequalities. These tools give lower and upper bounds on various connection probabilities for Bernoulli bond and site percolation on finite and infinite graphs. 

Most percolation results hold for specific graphs such as lattices or Cayley graphs. HK and vdBK inequalities are rare exceptions that apply to general graphs. Other inequalities include those proved by van den Berg, Kahn, and Häggström in \cite{BKa, BHK} and the author in \cite{G} and their corollaries. The recent work by Kozma and Nitzan \cite{KN} proposes a conjectured inequality for percolation on general graphs that would imply $\theta(p_c)=0$ for bond percolation on $\mathbb{Z}^d$, which is an old conjecture. They prove a plethora of corollaries of the inequalities above, aimed to prove their conjecture. The celebrated bunkbed conjecture can also be seen as an inequality for connection probabilities in a general graph.

The OSSS inequality is an inequality originating from the analysis of Boolean functions \cite{OSSS}. It was first applied to percolation models in \cite{DRTa} and was the key component in the proofs of several results about critical exponents \cite{Hut20, DRTb}. This allowed discussions about an ``OSSS method'' \cite{K20}. The method uses the concept of a (random) decision tree, that reveals the edges in an order dependent on the already revealed edges. 

% For example, one might start with a vertex $x$ of a graph $G = (V, E)$ and use the following DFS algorithm to query all edges with one end in the component of $x$ and only them.

% \begin{algorithm}
% \caption{Depth-First Search (DFS)}\label{alg:DFS}
% \begin{algorithmic}[1]
% \Procedure{DFS}{$G, x$} \Comment{Graph $G$ and source vertex $x$}
% \State $visited \gets \text{array of size } |V(G)| \text{ initialized with } false$
% \State $stack \gets \text{empty stack}$
% \State Push($stack, x$)
% \While{$stack$ is not empty}
%     \State $v \gets \text{Pop}(stack)$
%     \If{$visited[v]$ is $false$}
%         \State $visited[v] \gets \text{true}$
%         \For{each $u \in \text{neighbors of } v$}
%             \State Query $vu$
%             \If{$visited[u]$ is $false$}
%                 \State Push($stack, u$)
%             \EndIf
%         \EndFor
%     \EndIf
% \EndWhile
% \EndProcedure
% \end{algorithmic}
% \end{algorithm}

In \cite{GZ}, Zimin and the author have built several decision trees querying the edges in different order. We used them to build multiple percolation configurations. Their independence properties turn out to be enough to prove several new inequalities for connection probabilities for bond percolation in general graphs, including the proof that it is impossible for three vertices $a$, $b$, $c$ to be in the same cluster with probability $0 < p < 1$ and in three different clusters with probability $1 - p - \varepsilon$ for small enough $\varepsilon$. In this paper, we explore the dependencies between the percolation configurations obtained by the same tree and prove the decision tree generalizations of the HK and vdBK inequalities, as well as the inequality from \cite{GP} and the correct form of the inequality from \cite{R}. This allows us to prove new inequalities for connection probabilities in graph percolation. 

The structure of the paper is as follows. Section~\ref{sec:notation} introduces notation and the definitions for our method. Section~\ref{sec:HK} illustrates the method and proves the version of the HK inequality for decision trees. Section~\ref{sec:vdBK}, analogously, proves the version of the vdBK inequality. Section~\ref{sec:CS} utilizes the Cauchy--Schwarz inequality and finishes the groundwork for proving the inequalities we are interested in.

Put together, these results allow us to show in Section \ref{sec:DelfinoViti} the following inequality (see the full version in Theorem~\ref{thm:DV8}). In what follows, let $G = (V, E)$ be a locally finite connected simple graph and $
\P$ is the probability in a Bernoulli bond percolation model where each edge $e \in E$ is assigned a probability $p_e$ of being open.

\begin{theorem}[see Theorem~\ref{thm:DV8}]\label{thm:main}
Let $a$, $b$, $c$ be distinct vertices of graph $G$. Then
\begin{equation}\label{eq:8intro}
\P(abc)^2 \le 8\P(ab)\P(ac)\P(bc),
\end{equation}
where $\P(abc)$ is the probability that $a$, $b$ and $c$ are in the same percolation cluster.
\end{theorem}

The proof of this inequality combines together the decision tree versions of HK and vdBK inequalities as well as ideas from Section \ref{sec:CS}. This inequality can be seen as the $\sqrt{8}$ bound on the Delfino--Viti constant for every graph \cite{DV}. Moreover, when the graph $G$ is planar and $a$, $b$ and $c$ belong to the same face, we bring the constant $8$ in \eqref{eq:8intro} down to $2$.

Next, in Section \ref{sec:conj}, we prove the following technical asymmetric inequality on connection probabilities.
\begin{lemma}\label{lm:q2}
For vertices $a$, $b$ and $c$ in $G$ one has 
\begin{equation}\label{eq:q2}\P(a|b|c) + \P(a|b \cup a|c)^2 \ge \frac{\P(a|b|c)^2}{\P(a|b \cup b|c)} + \frac{\P(a|b|c)^2}{\P(a|c \cup b|c)},\end{equation}
where $\P(a|b|c)$ is the probability that $a$, $b$ and $c$ are in three different clusters and $v|u$ denotes the event that vertices $u$ and $v$ are in two different clusters.
\end{lemma}

This lemma allows us to resolve the following conjecture:

\begin{theorem}[formerly {\cite[Conj.~6.2]{GZ}}]\label{thm:conj2}
For $\varepsilon > 0$, there exists $\delta > 0$, such that 

$$\Big[\P(ab|c) < \delta\text{ and }\P(ac|b) < \delta\Big] \implies \Big[\P(abc) < \varepsilon\text{ or }\P(a|b|c)<\varepsilon\Big],$$
where $\P(a|b|c)$ is the probability that $a$, $b$ and $c$ are in three different clusters, $\P(abc)$ is the probability that $a$, $b$ and $c$ are in the same cluster, $\P(ab|c)$ is the probability that $a$ and $b$ are in the same cluster different from the cluster of $c$ and $\P(ac|b)$ is the probability that $a$ and $c$ are in the same cluster different from the cluster of $b$.
\end{theorem}
In fact, we believe the stronger Conjecture~\ref{conj:3} (\cite[Conj.~6.3]{GZ}). It describes the relation between four other connection events dependent on vertices $a$, $b$ and $c$ when $\P(ab|c) < \delta$. Substituting $\P(ac|b) < \delta$ into it recovers the Theorem~\ref{thm:conj2}.

In Section~\ref{sec:zipper}, we use decision trees to prove the main technical result (Main Lemma~\ref{thm:zipper}), that generalizes the proofs of the decision tree versions of the HK and vdBK inequalities. This general form makes it easier to prove various generalizations of the vdBK inequality. Additional implications include the positive mutual dependence for colored percolation (Theorem~\ref{thm:colored}), proved in \cite{GP} as well as an inequality from \cite{R}.

In Section \ref{sec:ab} we turn to inequalities concerning connection probabilities for just two points. We study the events of the form $ab^{\sq n}$, which stands for the existence of $n$ disjoint open paths between $a$ and $b$. It is easy to see from the vdBK inequality, that for every $n$ and $m$ we have
$$\P(ab^{\sq n+m}) \le \P(ab^{\sq n})\P(ab^{\sq m}).$$ 
In other words, $f(n) = \P(ab^{\sq n})$ is submultiplicative.

\begin{conjecture}\label{conj:log}
The function $f(n)$ is log-concave. Moreover, $\log(f(n))/n$ is decreasing.
\end{conjecture}

We provide a partial result in the direction of Conjecture \ref{conj:log}.

\begin{theorem}[cf. Theorem~\ref{thm:abklm}]\label{thm:ab23}
Let $G$ be planar. Suppose $a$ and $b$ belong to the same face. Then

\begin{equation}\label{eq:arms}
\P(ab^{\sq 3})^2 \le \P(ab^{\sq 2})^3.
\end{equation}
\end{theorem}

We also believe a stronger statement, see Conjecture~\ref{conj:pois}. 

% \pagebreak

\section{Definitions and notation}\label{sec:notation}
Throughout this paper, $G = (V, E)$ is a locally finite connected simple graph. We also assume that $a, b, c, d \in V$ are distinct vertices of $G$. A percolation configuration $C = (C(e): e \in E)$ on $G$ is a function from $E$ to $\{0, 1\}$. If $\omega_e = 1$, the edge $e$ is said to be open, otherwise $e$ is said to be closed. We deal with a bond percolation measure $\mu$ on the probability space $\Omega = \{0, 1\}^{E}$ of all percolation configurations. We assume that $\mu$ is a product measure, where each edge $e$ has its own probability $p_e$ of being open. This model is called the Bernoulli bond percolation. 
We have $\P$ refer to the probability of an event with respect to $\mu$. We also use the following notation from \cite{GZ}.

\begin{definition}
We denote by ``$v_{11}v_{12}\dots v_{1i_1}|v_{21}\dots v_{2i_2}|\dots|v_{n1}\dots v_{ni_n}$'' the event that the vertices $v_{11}$, \dots, $v_{1i_1} \in V$ belong to the same cluster, vertices $v_{21}$, \dots, $v_{2i_2}$ belong to the same cluster, \dots, vertices $v_{n1}$, \dots, $v_{ni_n}$ belong to the same cluster, and, moreover, these clusters are all different. By $\P(v_{11}v_{12}\dots v_{1i_1}|v_{21}\dots v_{2i_2}|\dots|v_{n1}\dots v_{ni_n})$ we denote the probability of this event in the underlying bond percolation. In particular, $\P(abc)$ denotes the probability that vertices $a, b, c \in V$ lie in the same cluster, and $\P(a|b|c)$ is the probability that $a$, $b$ and $c$ belong to $3$ different clusters.
\end{definition}

\begin{definition}
We call the event $A \subseteq 2^\Omega$ \emph{closed upward} if for every percolation configuration $C_1 \in A$ and every other configuration $C_2$ such that $C_1 \le C_2$ coordinatewise, one has $C_2 \in A$. For example, events $ab$ and $abc$ are closed upward.
\end{definition}

\begin{definition}
For two percolation configurations $C_1, C_2 \in \Omega$ and a set $S \subseteq E$ we denote by $C_1 \ra_S C_2$ 
the configuration that coincides with $C_1$ on $S$ and $C_2$ on its complement $\overline S$. 
\begin{equation}
C_1 \ra_S C_2(e) = \begin{cases}
C_1(e), \text{ if }e \in S,\\
C_2(e), \text{ otherwise.}
\end{cases}
\end{equation}
\end{definition}

The OSSS inequality introduced the concept of decision trees from computer science to percolation. A decision tree is an algorithm using a tree-like flowchart. Each node of the tree tests an edge of $G$ whether it is open or closed and uses this information to move to the next node. 
The tree \emph{decides} an event $A$ if for all the configurations leading to the same leaf node $L$, event $A$ is either simultaneously true or simultaneously false. Since we are working with probabilistic configurations, it can be beneficial to think that initially the states of edges are closed from us and an edge is \emph{revealed} when it is queried by the tree.

Until Section~\ref{sec:zipper}, we deal with the decision trees that accept two configurations $C_1$, $C_2$ and build a set $S \subset E$ based on them. Each node can make a decision based only on the edges revealed so far.

\begin{definition}\label{def:decisiontree1}
Let $G = (V, E)$ be finite. Let $T$ be a decision tree, where each node selects an edge, decides whether this edge goes to the set $S$ or $\overline{S}$ and reveals it in both $C_1$ and $C_2$. In this case, we say that the set $S = S(C_1, C_2)$ is \emph{built by $T$}.

Formally, a tree $T$ on a finite graph $G$ is an oriented network, containing nodes of two types -- decision nodes and leaf nodes. Each node $N$ contains an edge $e$ that it queries, a decision $D \in \{\text{``$S$''}, \text{``$\overline{S}$''}\}$, and the decision nodes moreover contain $4$ links to descendants indexed by $\{00, 01, 10, 11\}$. All nodes should be accessible via links from the initial node $N_0$ and the nodes on every path from $N_0$ should query pairwise distinct edges. The set $S(C_1, C_2)$ is then built using Algorithm \ref{alg:sbyt}.

\begin{algorithm}[hbt]
\small
\caption{Building Set $S$ by Decision Tree $T$}\label{alg:sbyt}
\begin{algorithmic}[1]
\Procedure{BuildSet}{$T, C_1, C_2$}
\State $S \gets \emptyset$
\State $N \gets N_0$ \Comment{Start from the root node $N_0$ of the decision tree $T$}
\While{$N$ is a decision node}
\State $e \gets \text{edge queried by } N$
\If{$\text{decision of } N \text{ is ``S''}$}
\State $S \gets S \cup \{e\}$
\If{$C_1(e) = 1$ and $C_2(e) = 1$}
\State $N \gets N_{11}$ \Comment{Both configurations have edge $e$ open}
\ElsIf{$C_1(e) = 1$ and $C_2(e) = 0$}
\State $N \gets N_{10}$ \Comment{Configuration $C_1$ has edge $e$ open and $C_2$ has it closed}
\ElsIf{$C_1(e) = 0$ and $C_2(e) = 1$}
\State $N \gets N_{01}$ \Comment{Configuration $C_1$ has edge $e$ closed and $C_2$ has it open}
\Else
\State $N \gets N_{00}$ \Comment{Both configurations have edge $e$ closed}
\EndIf
\EndIf
\EndWhile \Comment{Now $N$ is a leaf node}
\State $e \gets \text{edge queried by } N$
\If{$\text{decision of } N \text{ is ``S''}$}
\State $S \gets S \cup \{e\}$
\EndIf
\State \Return $S$
\EndProcedure
\end{algorithmic}
\end{algorithm}
\end{definition}

\begin{example}
    Assume that $T$ first reveals the edges adjacent to some specific vertex $a$. Then $T$ reveals the edges connected to the vertices connected to $a$ via the revealed open edges and so on, until all edges with one end in the cluster of $a$ are revealed. This is the breadth-first search (BFS) algorithm, as opposed to the depth-first search (DFS) Algorithm~\ref{alg:DFS_tree}. Assume that $T$ puts all the revealed edges in $S$. Then the set $S$ built by $T$ is the set of edges with at least one end in the cluster of $a$. 
\end{example}
For a more detailed and visual example, see \cite[Figure~1]{GZ}.

\section{HK inequality for decision trees}\label{sec:HK}

The key lemma used by Zimin and the author is the following independence result:

\begin{lemma}[{\cite[Lemma~4.2]{GZ}}]\label{lemma:independence}
Let $G$ be finite. Let $S(C_1, C_2)$ be built by some decision tree. Then $C_1 \ra_S C_2$ is independent of $C_2 \ra_S C_1 = C_1 \ra_{\bar S} C_2$ and both are distributed as $\mu$.
\end{lemma}

This lemma alone is enough to justify some inequalities of the new type. Moreover, it turns out that many classic correlation inequalities can be transferred to work with the events of type $C_1 \ra_S C_2$. First, we prove a positive correlation result. When $S$ is the set of all edges $E$, this result gives the usual HK inequality.

The HK inequality was independently discovered by Harris \cite{H60} in the context of percolation and Kleitman \cite{K66} in the context of set families. It ensures that every two closed upward events have a nonnegative correlation. The HK inequality was later generalized to the broader class of measures by Fortuin, Kasteleyn and Ginibre in \cite{FKG}, so it is often also called the FKG inequality.

\begin{theorem}[Decision tree HK inequality] \label{thm:HK}
Let $G$ be finite. Let $S(C_1, C_2)$ be built by some decision tree. Assume $A$ and $B$ are some events in $\Omega$ closed upward. Then
\begin{equation*}
\P(C_1 \in A, C_1 \ra_S C_2 \in B) \ge \P(C_1 \in A)\P(C_1 \ra_S C_2 \in B) = \mu(A)\mu(B).
\end{equation*}
\end{theorem}

\begin{proof}
We use induction on the number of nodes in $T$ with $D(N) = \text{``$S$''}$. In case when $T$ always sends edges to $\overline{S}$, the inequality becomes an equality. Otherwise, consider all nodes of $T$ with $D(N) = \text{``$S$''}$ and choose out of them a node $N$ with edge $e(N)$ lying on the lowest level. Then, all descendants of $N$ send their edges to $\overline{S}$. Consider the tree $T'$ building a set $S'$ that coincides with $T$ in all nodes except for $N$, with the distinction that $D(N')=\text{``$\bar{S'}$''}$. Now, let $\Omega'$ be the probability space for all edges except for $e$. For each configuration $C$ in $\Omega'$ there are two ways to extend it to a configuration on $\Omega$, namely $C^+$ where the edge $e$ is open and $C^-$ where $e$ is closed. 

Now assume that the restriction of $C_1 \times C_2$ to $\Omega'\times\Omega'$ is fixed. We will get the induction step inequality
\begin{equation}\label{eq:HKstep}
\P(C_1 \in A, C_1 \ra_S C_2 \in B) \ge \P(C_1 \in A, C_1 \ra_{S'} C_2 \in B)
\end{equation}
by summing over all restrictions.
Since all edges, that were not queried until node $N$, are sent to $\overline{S}$, the configuration $C_1 \ra_S C_2$ is defined up to edge $e$. We will call the possible configurations $C_3^+$ and $C_3^-$.

$$C_3^+ = C_1^+ \ra_S C_2^- = C_1^+ \ra_S C_2^+ = C_1^- \ra_{S'} C_2^+ = C_1^+ \ra_{S'} C_2^+,$$
and
$$C_3^- = C_1^- \ra_S C_2^- = C_1^- \ra_S C_2^+ = C_1^- \ra_{S'} C_2^- = C_1^+ \ra_{S'} C_2^-.$$
Moreover, since $B$ is closed up, we have three possibilities: both $C_3^+$ and $C_3^-$ belong to $B$, none of them belong to $B$ or only $C_3^+$ does. In the first case, 
$$\P(C_1 \in A, C_1 \ra_S C_2 \in B) = \P(C_1 \in A) = \P(C_1 \in A, C_1 \ra_{S'} C_2 \in B).$$
In the second case, 
$$\P(C_1 \in A, C_1 \ra_S C_2 \in B) = 0 = \P(C_1 \in A, C_1 \ra_{S'} C_2 \in B).$$
Finally, the third case is split into $3$ subcases as well. If both $C_1^+$ and $C_1^-$ belong to $A$ or do not belong to $A$, the induction step is still trivial. The only nontrivial subcase is when $C_1^+$ belongs to $A$, but $C_1^-$ does not. In this case, $C_1 \in A, C_1 \ra_S C_2 \in B$ means that $e$ is open in $C_1$. At the same time, $C_1 \in A, C_1 \ra_{S'} C_2 \in B$ means that $e$ is open in both $C_1$ and $C_2$. So, inequality \eqref{eq:HKstep} holds in this case. Finally, summing over all the restrictions on $\Omega'\times\Omega'$ we prove the inequality \eqref{eq:HKstep} and complete the induction.
\end{proof}

% \begin{remark}\label{rem:influence}
% The proof shows a stronger statement. For node $N$, let $\delta(N)$ be the probability that $T$ visits node $N$. In the spirit of the OSSS inequality, we introduce the \emph{influences} $\Inf_N(A)$ and $\Inf_N(B)$ as the probability that $e(N)$ is \emph{pivotal} for $A$ in $C_1$ and for $B$ in $C_1 \ra_S C_2$ given that $T$ visits $N$. Then 
% \end{remark}

% TODO: write the formula for the difference of the sides, use it to prove $\P(abc|d) \ge \P(ab|d)\P(ac|d)$ (weak BHK).

\section{Decision tree vdBK inequality}\label{sec:vdBK}

The counterpart to the HK inequality is the vdBK inequality, which can be thought of as a sort of negative correlation inequality. For decision trees, these inequalities can beautifully work together, providing simple lower and upper bounds on the probabilities of events dependent on $S$.

\begin{definition}
For a space $\Omega = \prod_{i=1}^n \Omega_i$, a \emph{witness} of an event $A$ in a configuration $C$ is a subset $I$ of $[n]$, such that for any configuration $C'$ that has the same coordinate as $C$ for all $\Omega_i$ for $i \in I$ one has $C' \in A$.    

One defines the \emph{disjoint occurrence of $A$ and $B$} denoted by $A \sq B$ as 
\begin{multline*}
A \sq B := \{ C \in \Omega, \text{ s.t. there exist }I, J \subset [n]\\\text{ s.t. $I$ is a witness of $A$ in $C$, $J$ is a witness of $B$ in $C$ and }I\cap J = \varnothing\}.
\end{multline*}
\end{definition}

The natural generalization to the decision trees involves the set $S$.

\begin{definition}
For the decision trees setup, the disjoint occurrence $A \sq_{S} B$ is given by
\begin{multline*}
A \sq_S B := \{ C_1, C_2 \in \Omega, \text{ s.t. there exist }I, J \subset [n]\\\text{ s.t. $I$ is a witness of $A$ in $C_1$, $J$ is a witness of $B$ in $C_1 \ra_S C_2$ and }I\cap J \subseteq \overline{S}\}.
\end{multline*}
\end{definition}

For $S=E$, this definition turns into the usual disjoint occurrence of $A$ and $B$ in $C_1$. For $S=\varnothing$, the event $A \sq_S B$ coincides with $A \times B$. 

\begin{theorem}[Decision tree vdBK inequality]\label{thm:vdBK}
Let $G$ be finite.
Let the decision tree $T$ build a set $S(C_1, C_2)$ and $A$ and $B$ be two closed upward events. Then 
$P(A \sq_S B) \le P(A)P(B).$
\end{theorem}
\begin{proof}
As in the proof of Theorem~\ref{thm:HK}, we induct on the number of nodes in $T$ sending their edge to $S$. Again, $N$ is such a node lying on the lowest level, $e$ is an edge $N$ sends to $S$, $T'$ coincides with $T$ in all nodes except for $e$ and $\Omega'$ is the probability space for all edges except for $e$.

Again, we assume that the restriction of $C_1 \times C_2$ to $\Omega'\times\Omega'$ is fixed and prove the inequality
\begin{equation}\label{eq:vdBKstep}
\P\big((C_1, C_2) \in A \sq_S B\big) \le \P\big((C_1, C_2) \in A \sq_{S'} B\big)
\end{equation}
for each restriction.

Assume that $(C_1, C_2) \in A \sq_{S} B$, but $(C_1, C_2) \not\in A \sq_{S'} B$. Since $A$ and $B$ are closed upward, that means that $e$ is open in $C_1$ and the set $J$ used in the witness for $(C_1, C_2) \in A \sq_S B$ contains $e$. Moreover, one can see that $C_1$ has $e$ closed and $C_2$ has $e$ open. Then notice that the configuration $(C_1^-, C_2^+)$ has the same probability as $(C_1, C_2)$, has the same restriction to $\Omega'\times\Omega'$, but it would, in contrast, lie in $A \sq_{S'} B$, but not $A \sq_{S} B$. 
Indeed, if $(I, J)$ was the witness for $(C_1, C_2) \in A \sq_{S} B$, then one can assume $I$ does not contain $e$ since $A$ is closed upward. So, $(I, J)$ would witness $(C_1^-, C_2^+) \in A \sq_{S'} B$. Also, assume $(I', J')$ witnesses $(C_1^-, C_2^+) \in A \sq_{S'} B$. Then again by upward closeness we assume that $I'$ does not contain $e$, but $J'$ does and so the pair $(I' \cup \{e\}, J' \setminus \{e\})$ is the witness for $(C_1, C_2) \in A \sq_S B$. Thus, for each restriction, the inequality \eqref{eq:vdBKstep} holds, and so the induction step is complete.    
\end{proof}

% TODO: inequality $\Big(\P(abcd) - \P(ab)\P(cd) \Big)\P(b|c) \le \P(ad|b) + \P(ad|c)$ for outerplanar graphs.

\section{Approach via Cauchy--Schwarz inequality}\label{sec:CS}

%If the decision trees in the previous sections stopped before querying all edges, the remaining edges were automatically put in $\overline{S}$. But there was no incentive for the decision trees not to go till the very end. 

By Definition~\ref{def:decisiontree1}, the decision tree can have some leaf nodes such that not all edges are queried on the path leading to them. According to Algorithm~\ref{alg:sbyt}, such edges are not assigned to $S$ and therefore are assigned to $\overline{S}$. If we replace some of the leaf nodes with the subtrees, we will get a new tree. We say that the new tree is a continuation of the old tree.

\begin{definition}
We say that the decision tree $T_2$ \textit{continues} decision tree $T_1$, when $T_1$ is a subset of nodes of $T_2$, where with each node $N\in T_2$, $T_1$ includes all its ancestors. So all nodes of $T_1$ put their edges in $S$ or $\overline{S}$ the same way as their counterparts in $T_2$. In particular, if $T_1$ builds the set $S_1$ and $T_2$ builds the set $S_2$, then $S_1(C_1, C_2) \subseteq S_2(C_1, C_2)$. We also say that the decision tree $T$ \emph{decides} an event $A \subseteq \Omega$, when for every leaf $L$ of $T$, the set of edges revealed on the path from the root to $L$ witnesses either the event $C_1 \in A$ or the event $C_1 \not\in A$.
\end{definition}
By this definition, $T_2$ is able to decide finer events than $T_1$. % Also, all trees that we will consider take only information from $C_1$ to make their decisions. 
% We will use this fact in the next theorem.

\begin{theorem}\label{thm:CS}
Let $G$ be finite. Let $T_1$ and $T_2$ be decision trees for events $C_1 \in A$ and $C_1 \in B$ respectively, such that $T_2$ continues $T_1$ and $B \subset A$ is an intersection of $A$ with an increasing or decreasing event in $\Omega$. In addition, assume that all nodes of $T_1$ send the edges to $S$. Then
\begin{equation}\label{eq:CS}
\P(C_1 \in B, C_1 \ra_{S_2} C_2 \in B) \ge \frac{\P(B)^2}{\P(A)}.
\end{equation}
\end{theorem}
\begin{proof}
Let $\delta(N)$ be %, as in Remark~\ref{rem:influence}, 
the probability that $T_1$ visits node $N$. We cal; it the \textit{influence} of $N$. It is easy to see that the sum of the influences of the leaves of $T_1$ is equal to $1$. Then we can write the probabilities of $A$ and $B$ as a sum over the leaves of $T_1$. Denote the set of leaves of $T_1$ where $T_1$ concludes $A$ by $X$. Then, 

\begin{equation}\label{eq:CS1}
\P(A) = \sum_{N \in X} \delta(N).
\end{equation}

Since $B$ is a subset of $A$, we can break the probability of $B$ by which node of $X$ it came through in $T_2$.
\begin{equation}\label{eq:CS2}
\P(B) = \sum_{N \in X} \delta(N) \P(B~|~ T_2 \text{ goes through }N).
\end{equation}

Finally, since $T_1$ only sent the vertices to $S$, for each $N \in X$ we can consider the subtree $T_N$ of $T_2$ after the node $N$ and apply Lemma \ref{lemma:independence} there to conclude that the conditional distributions of $C_1$ and $C_1 \ra C_2$ coincide. Since $B$ is an intersection of $A$ with a monotone event, $B$ is monotone in $T_N$. By Theorem~\ref{thm:HK} applied to $T_N$ and the events $C_1 \in B$ and $C_1 \ra_{S_2} C_2 \in B$ we get the representation
\begin{equation}\label{eq:CS3}
\P(C_1 \in B, C_1 \ra_{S_2} C_2 \in B) \ge \sum_{N \in X} \delta(N) \P(B~|~ T_2 \text{ goes through }N)^2.
\end{equation}

Let us enumerate the nodes in $X$ and consider the vectors $\overrightarrow{v}$ and $\overrightarrow{w}$ indexed by $X$:
$$\overrightarrow{v} = \{\sqrt{\delta(N)} \}_{N \in X},~ \overrightarrow{w} = \{\sqrt{\delta(N)} \P(B~|~ T_2 \text{ goes through }N)\}_{N \in X}.$$

Finally, applying the Cauchy--Schwarz inequality to these vectors and using equations \eqref{eq:CS1}, \eqref{eq:CS2} and \eqref{eq:CS3}, we get \eqref{eq:CS}.
\end{proof}

\begin{corollary}
Assume that some tree $T$ first queries the edges from the component of $a$ in $C_1$ and puts them in $S$. Then, regardless of what it does further,

\begin{equation}\label{eq:frac1}
\P(C_1 \in a|b|c, C_1\ra_{S} C_2 \in a|b|c) \le \frac{\P(a|b|c)^2}{\P(a|b \cup a|c)}
\end{equation}
and
\begin{equation}\label{eq:frac2}
\P(C_1 \in a|bc, C_1\ra_{S} C_2 \in a|bc) \le \frac{\P(a|bc)^2}{\P(a|b \cup a|c)}.
\end{equation}
\end{corollary}
\begin{proof}
Indeed, let $T_1$ be the subtree of $T$ cut at the moment where $T$ queries all the edges from the component of $a$. By Theorem~\ref{thm:CS} applied to the trees $T_1$ and $T$ and the decreasing events $A = a|b \cup a|c$ and $B = a|b|c$, we get Equation \eqref{eq:frac1}.

To obtain equation \eqref{eq:frac2}, consider the same trees and $A = a|b \cup a|c$ and $B = A \cap bc = a|bc$.
% \begin{align*}
% \P(C_1 \in a|bc&,~C_1\ra_{S} C_2 \in a|bc) 
% \\ &= \P(T_1\text{ concludes }a|b\cup a|c) - \P(C_1 \in a|b|c\text{ or } C_1\ra_{S} C_2 \in a|b|c) 
% \\ &= \P(a|b\cup a|c) - \Big( 2\P(a|b|c)  - \P(C_1 \in a|b|c\text{ and } C_1\ra_{S} C_2 \in a|b|c)\Big) 
% \\ &= \P(a|b\cup a|c) - 2\P(a|b|c) + \frac{\P(a|b|c)^2}{\P(a|b \cup a|c)} 
% \\ &= \frac{\P(a|b\cup a|c)^2 - 2\P(a|b|c)\P(a|b\cup a|c) + \P(a|b|c)^2}{\P(a|b \cup a|c)} = \frac{\P(a|bc)^2}{\P(a|b \cup a|c)}.
% \end{align*}
\end{proof}

\section[Delfino--Viti constant is less than square root of 8]{Delfino--Viti constant for general graphs is less than $2\sqrt{2}$}\label{sec:DelfinoViti}

If graph $G$ is planar, we can say more about bond percolation on it. First, it allows for some graph simplifications like the star--triangle transformation, the effect of which on bond percolation is explained in \cite{Wi}. In the context of the bunkbed conjecture, the star--triangle transformations were also used by Linusson in \cite{Lin} (See also \cite{Lin2}). 

What is more, the assumption of planarity allows the decision trees to use the right-hand and left-hand rules for solving mazes: put your right (left) hand on the wall and keep it there until you find an exit. In our setup, it means the following: query the edges in the order of the depth-first search (DFS) and in each vertex choose the node visiting order starting from where you came, right to left (left to right). When the initial node is on the outer face, we choose the visiting order right to left (left to right) starting from the outer face. The DFS with the right order together with the theorems from the previous sections gives the following results.

\begin{theorem}\label{thm:planar}
Let $G$ be a finite planar graph and $a, b, c$ lie on the outer face. Then 
\begin{equation*}
\P(abc)^2 \le 2\P(ab)\P(bc)\P(ac).
\end{equation*}
\end{theorem}

For nonplanar graphs, there are two ways to prove a weaker inequality.

\begin{theorem}[cf. Theorem~\ref{thm:main}]\label{thm:DV8}
For Bernoulli bond percolation on a graph $G$ with vertices $a$, $b$, $c$ one has 
\begin{equation}\label{eq:with8}
\P(abc)^2 \le 8\P(ab)\P(ac)\P(bc)
\end{equation}
and 
\begin{equation}\label{eq:withunion}
\P(abc)^2 \le 2\P(ab \cup ac)^2\P(bc).
\end{equation}
\end{theorem}

\subsection{Proof of Theorem~\ref{thm:planar}}

Assume that $a$, $b$ and $c$ lie on the outer face in this clockwise order. Let us build a decision tree $T$ using the DFS decision tree from Algorithm \ref{alg:DFS_tree} as $$T = DFS\_Decision\_tree(G, a, visited_e=\varnothing, S=\varnothing, \text{right-hand rule}, decision),$$ where $decision$ returns ``$S$'' if $c$ is not yet visited and ``$\overline{S}$'' otherwise.

\begin{algorithm}
\small
\caption{DFS Decision Tree}\label{alg:DFS_tree}
\begin{algorithmic}[1]
\Procedure{DFS\_Decision\_tree}{$G, x, \&visited\_e, \&S, order\_decider, decision$}
\Comment{Graph $G$, source vertex $x$, reference to list of visited edges $visited_e$, reference to set $S$, function $order\_decider$ deciding the order of neighbors, function $decision$ outputting $D \in \{\text{``$S$''}, \text{``$\overline{S}$''}\}$}
\State $visited_v \gets \text{array of size } |V(G)| \text{ initialized with } false$
\State $stack \gets \text{empty stack}$
\State Push($stack, x$)
\While{$stack$ is not empty}
\State $v \gets \text{Pop}(stack)$
\If{$visited_v[v]$ is $false$}
\State $visited_v[v] \gets \text{true}$
\State $neighbors \gets order\_decider(\{\textit{neighbors of } v\}, v, stack, visited_e, visited_v)$ 
\For{each $u \in neighbors$}
\If{$visited_e[vu]$ is $false$}
\State Put $vu$ in $decision(vu, v, stack, visited_e, visited_v)$ \Comment{Decision node}
\If{$visited[u]$ is $false$}
    \State Push($stack, u$)
\EndIf
\EndIf
\EndFor
\EndIf
\EndWhile
\EndProcedure
\end{algorithmic}
\end{algorithm}

So our tree would first use a right-hand rule to build a route from $a$ to $c$ in $C_1$ and add it to $S$. With probability $\P(a|c)$, the tree queries the whole component of $a$ in $C_1$, since it does not contain $c$. If this is the case, we stop $T$. Suppose that this is not what happens. Then we effectively stop $T$ anyway after reaching $c$, adding the rest of the edges to $\overline{S}$. The vertices visited by $T$ until this moment then will be the vertices to the right of the rightmost path by open edges from $a$ to $c$. Denote this path by $P$. All edges of $P$ belong to $S$ as well as all edges to the right of it.

Let $T'$ be the continuation of $T$ that reveals the remaining edges and puts them in $\overline{S}$. Then $T'$ can decide $C_1 \in abc$. We are interested in $\P(C_1 \in abc, C_1\ra_S C_2 \in abc)$. Applying Theorem~\ref{thm:CS} for $T_1 = T$, $T_2 = T'$ and the increasing events $A = ac$ and $B = abc$, we get
\begin{equation}\label{eq:outerplanarCS}
\P(C_1 \in abc, C_1\ra_S C_2 \in abc) = \frac{\P(abc)^2}{\P(ac)}.
\end{equation} 
On the other hand, assume $C_1 \in abc$ and $C_1\ra_S C_2 \in abc$ occurred. Then there is a path from $b$ to $a$ in $C_1$. The first time this path intersects with $P$, it must do it from the left side of $P$ and so all the edges before this point belong to $\overline{S}$. Denote this initial fragment of the path by $P_1$. Similarly, we consider a path from $b$ to $a$ in $C_1\ra_S C_2$ and its initial fragment before meeting $P$ and denote it by $P_2$. So there are paths $P_1$ and $P_2$, consisting of edges from $\overline{S}$ such that they both connect $b$ to $P$ and the edges of $P_1$ are open in $C_1$ and the edges of $P_2$ are open in $C_2$. 

Let $v_1$ be the vertex in $P$ that is connected to $b$ through $P_1$ and $v_2$ be the vertex on $P$ connected to $b$ via $P_2$. If on the path $P$ the vertex $v_1$ lies closer to $a$ than $v_2$, denote the segments of path $P$ by $a \leadsto_P v_1$, $v_1 \leadsto_P v_2$, $v_2 \leadsto_P c$. Then the paths $a \leadsto_P v_1 \cup P_1$ and $v_2 \leadsto_P c \cup P_2$ are witnesses of the events $ab$ and $bc$ respectively and their intersection belongs to $\overline{S}$. It shows that $(C_1, C_1\ra_S C_2) \in (ab \sq_S bc)$. On the contrary, if $v_2$ is closer to $c$ than $v_1$, the paths $v_1 \leadsto_P c \cup P_1$ and $v_2 \leadsto_P a \cup P_2$ are the witnesses that prove $(C_1, C_1\ra_S C_2) \in (bc \sq_S ac)$. Altogether, we get the estimate
\begin{equation}\label{eq:vdbkapplication}
\P(C_1 \in abc, C_1\ra_S C_2 \in abc) \le \P(ab \sq_S bc \cup bc \sq_S ab).
\end{equation}

By Theorem~\ref{thm:vdBK}, the right side is bounded from above by $2\P(ab)\P(bc)$. Combining with \eqref{eq:outerplanarCS}, we get  $\frac{\P(abc)^2}{\P(ac)} \le 2\P(ab)\P(bc)$, which is equivalent to the theorem statement.
\qed

\begin{remark}
By Theorem~\ref{thm:strongBK2} one is able to improve over the vdBK estimate of the right side in \eqref{eq:vdbkapplication} and show $\frac{\P(abc)^2}{\P(ac)} \le 2\P(ab)\P(bc) - \P(abc)^2$.
\end{remark}

\subsection{Proof of Theorem~\ref{thm:DV8}}
Assume $G$ is finite. We first prove the inequality \eqref{eq:withunion}. Let tree $T$ perform a DFS starting with the vertex $a$ and put the edges it meets in $S$. With probability $\P(a|b \cup a|c)$, the tree queries the whole component of $a$ in $C_1$ since it does not contain $b$ or $c$. After reaching $b$ or $c$, the tree $T$ stops (and so puts the rest of the edges in $\overline{S}$). 

Backtracking the DFS order leaves us with a path $P$ from $a$ to either $b$ or $c$. Note that $T$ queries all the edges of the path $P$ and puts them to $S$. Let $Q$ be the set of vertices visited by the DFS that are not in $P$. The set $S$ witnesses that all vertices from $Q$ are connected to $a$. Also, since vertices from $Q$ do not belong to the final path $P$, it means that the DFS queried all edges from $Q$ before backtracking and so all of these edges, including those closed in $C_1$, belong to~$S$. 

Consider the tree $T'$ continuing $T$ that reveals the remaining edges putting them in $\bar{S}$ and so is able to decide the event $abc$. 
Now, by Theorem~\ref{thm:CS} applied to the trees $T$ and $T'$ and the events $ab \cup ac$ and $abc$, 
$$\P(C_1 \in abc, C_1\ra_S C_2 \in abc) \ge \frac{\P(abc)^2}{\P(ab \cup ac)}.$$
On the other hand, as in the previous proof, $\P(C_1 \in abc, C_1\ra_S C_2 \in abc)$ is bounded from above by $2\P(ab \cup ac)P(bc)$, by Theorem~\ref{thm:vdBK}. It proves the inequality in the form \eqref{eq:withunion}.

Now we prove \eqref{eq:with8}.
Denote by $R_b$ the event that $T$ connects $a$ to $b$ and by $R_c$ the event that $T$ connects $a$ to $c$. It is easy to see that $\P(R_b) + \P(R_c) = \P(ab \cup ac)$ and further $\P(R_b) \le \P(ab)$ and $\P(R_c) \le \P(ac)$. Also, by Theorem~\ref{thm:CS} we have
\begin{equation*}
\P(C_1 \in R_b \cap abc,  C_1\ra_{S} C_2 \in R_b \cap abc) \ge \frac{\P(R_b \cap abc)^2}{\P(R_b)}
\end{equation*}
and
\begin{equation*}
\P(C_1 \in R_c \cap abc, C_1\ra_{S} C_2 \in R_c \cap abc) \ge \frac{\P(R_c \cap abc)^2}{\P(R_c)}.
\end{equation*}

On the other hand, we can estimate the LHS of the two inequalities, as in the previous proof, using Theorem~\ref{thm:vdBK}:

\begin{equation*}
\P(C_1 \in R_b \cap abc, C_1\ra_{S} C_2 \in R_b \cap abc) \le 2\P(ac)\P(bc)
\end{equation*}
and
\begin{equation*}
\P(C_1 \in R_c \cap abc, C_1\ra_{S} C_2 \in R_c \cap abc) \le 2\P(ab)\P(bc).
\end{equation*}

Combining these four inequalities, we get 
% \begin{align*}
% \P(abc) \ge& ~\P(R_b \cap abc) + \P(R_c \cap abc) \\
% \ge& \sqrt{2\P(ab)\P(ac)\P(bc)} + \sqrt{2\P(ab)\P(ac)\P(bc)} = 2\sqrt{2}\sqrt{\P(ab)\P(ac)\P(bc)},
% \end{align*}
\begin{multline*}
    \P(abc) = ~\P(R_b \cap abc) + \P(R_c \cap abc) \\
    \le \sqrt{2\P(ac)\P(bc)\P(R_b)} + \sqrt{2\P(ab)\P(bc)\P(R_c)} \le 2\sqrt{2\P(ab)\P(ac)\P(bc)},
\end{multline*}
which implies \eqref{eq:with8}. This completes the proof for finite $G$. For infinite $G$, both \eqref{eq:withunion} and \eqref{eq:with8} follows by passing to the limit.
\qed

\begin{remark}
By the specific randomized ordering choice of vertices in DFS (called PDFS in \cite{GP24}), one can ensure that $$\P(R_b \cap abc) = \P(R_c \cap abc) = \frac{\P(abc)}{2}.$$
Not only does it the proof more straightforward, it also gives a slightly tighter bound
$$\P(abc) \le \sqrt{2\P(ac)\P(bc)\left(\P(ab|c)+\frac{\P(abc)}{2}\right)} + \sqrt{2\P(ab)\P(bc)\left(\P(ac|b)+\frac{\P(abc)}{2}\right)}.$$
%It makes the proof more straightforward. Unfortunately, this does not improve the constant in \eqref{eq:with8}.
\end{remark}

\subsection{Implications}
% box-crossing (RSW) inequalities
%RSW theory

An interesting case emerges when one applies this to the critical mode of percolation on $\mathbb{Z}^2$. From the box-crossing (RSW) inequalities (see \cite[Section 11.7]{Gr}) and the HK inequality, one can show that $\frac{\P(abc)}{\sqrt{\P(ab)\P(ac)\P(bc)}}$ is bounded. Our method allows for better estimates on this bound.
According to the conjectured integral formula from \cite{DV}\footnote{The proof of this formula for critical site percolation on triangular lattice was recently announced by Morris Ang,Gefei Cai  Xin Sun and Baojun Wu (personal communication, 10 Aug 2024)}, for $\mathbb{Z}^2$ this quantity converges to approximately $1.022$ as $a$, $b$ and $c$ tend from each other. This number is consistent with our upper bound of $2\sqrt{2}$. 

Moreover, as per the earlier result in \cite{SZK, BI}, for the bond percolation on the upper half-plane, if $a$, $b$ and $c$ lie on the real line, the scaling limit of $\frac{\P(abc)}{\sqrt{\P(ab)\P(ac)\P(bc)}}$ converges to 
$$\frac{2^{\frac72}\pi^\frac52}{3^\frac34\Gamma(\frac13)^\frac92}\approx 1.02992\dots$$
In this case, our Theorem~\ref{thm:planar} is applicable and gives a consistent upper bound of $\sqrt{2}$.

For the supercritical mode, denote by $\theta$ the density of the infinite cluster. Then equation \eqref{eq:with8} tends to $\theta^6 \le 8\theta^6$ as $a$, $b$ and $c$ tend away from each other.

In the non-integrable cases, our inequality still leads to an inequality on the three-point exponent. Denote by $D(a, b, c)$ the maximum distance between $a$, $b$ and $c$: 
$$D(a, b, c) = \max(D(a, b), D(a, c), D(b, c)).$$

\begin{corollary}
Let $G$ be a vertex-transitive infinite graph and $C$ and $\alpha$ be the constant such that $\P(ab) < CD(a, b)^\alpha$. Then $$\P(abc) < (2C)^\frac{3}{2}D(a, b, c)^{\frac{3}{2}\alpha}.$$
\end{corollary}

Note that for $G$ being a two-dimensional lattice in the critical mode, \cite{LSW}, assuming conformal invariance, establishes that $\P(ab) = D(a, b)^{-2\eta+o(1)}$, where $\eta = \frac{5}{48}$ is the one-arm exponent. The $\P(abc)$, in turn, grows as $D(a, b, c)^{-3 \eta+o(1)}$, which coincides with our bound. Note that the conformal invariance is only known for site percolation on the triangular lattice as per the celebrated result of Smirnov\cite{Sm}.

\section[Proof of Theorem~\ref{thm:conj2}]{Proof of Theorem~\ref{thm:conj2}}\label{sec:conj}

\subsection{Proof of Lemma~\ref{lm:q2}}
We slightly modify the proof of Theorem~4.2 in \cite{GZ} and use the better bound from Theorem~\ref{thm:CS}. Assume $G$ is finite. The following lemma was a keystone in the proof:

\begin{lemma}
Let the decision function $\mathcal{S}$ always output ``$S$'' and the decision function $\bar{\mathcal{S}}$ always output ``$\overline{S}$''. Let $S_1$, $S_2$ and $S_3$ be given by the decision trees from Figure \ref{fig:S1S2S3}.

\begin{figure}
\centering
\begin{algorithm}[H]\small
\caption{$S_1$ Decision Tree}
\begin{algorithmic}[1]
\Procedure{DFS\_Decision\_tree}{$G, a, b, c$}
%\Comment{Graph $G$, vertices a, b, c$}
\State $visited_e \gets \text{array of size } |E(G)| \text{ initialized with } false$, 
$S \gets \varnothing$
\State $DFS\_Decision\_tree(G, c, visited_e, S, id, \mathcal{S})$
\State $DFS\_Decision\_tree(G, a, visited_e, S, id, \bar{\mathcal{S}})$
\State $DFS\_Decision\_tree(G, b, visited_e, S, id, \mathcal{S})$
\State \Return $S$
\EndProcedure
\end{algorithmic}
\end{algorithm}

\begin{algorithm}[H]\small
\caption{$S_2$ Decision Tree}
\begin{algorithmic}[1]
\Procedure{DFS\_Decision\_tree}{$G, a, b, c$}
%\Comment{Graph $G$, vertices a, b, c$}
\State $visited_e \gets \text{array of size } |E(G)| \text{ initialized with } false$, 
$S \gets \varnothing$
\State $DFS\_Decision\_tree(G, b, visited_e, S, id, \mathcal{S})$
\State $DFS\_Decision\_tree(G, a, visited_e, S, id, \bar{\mathcal{S}})$
\State $DFS\_Decision\_tree(G, c, visited_e, S, id, \mathcal{S})$
\State \Return $S$
\EndProcedure
\end{algorithmic}
\end{algorithm}

\begin{algorithm}[H]\small
\caption{$S_3$ Decision Tree}
\begin{algorithmic}[1]
\Procedure{DFS\_Decision\_tree}{$G, a, b, c$}
%\Comment{Graph $G$, vertices a, b, c$}
\State $visited_e \gets \text{array of size } |E(G)| \text{ initialized with } false$, 
$S \gets \varnothing$
\State $DFS\_Decision\_tree(G, a, visited_e, S, id, \bar{\mathcal{S}})$
\State $DFS\_Decision\_tree(G, b, visited_e, S, id, \mathcal{S})$
\State $DFS\_Decision\_tree(G, c, visited_e, S, id, \mathcal{S})$
\State \Return $S$
\EndProcedure
\end{algorithmic}
\end{algorithm}
\caption{Trees building $S_1, S_2, S_3$}
\label{fig:S1S2S3}
\end{figure}

Then if $C_1 \in a|b|c$ and $C_1\ra_{S_3}C_2\in ab \cup ac$, one has $C_1\ra_{S_1}C_2\in ab$ or $C_1\ra_{S_2}C_2\in ac$.

\end{lemma}

We use this lemma, and as in \cite{GZ}, we bound the probability of the first event from above by 
\begin{multline*}
\P(C_1 \in a|b|c\text{ and }C_1\ra_{S_3}C_2\in ab \cup ac) \\\le \P(ab \cup ac)\P(a|b \cap a|c) - \P(a|bc) = \P(a|b|c) - \P(a|b \cap a|c)^2.
\end{multline*}

Now, using Theorem~\ref{thm:CS} we can get a better lower bound for the probabilities of the two latter events. Indeed, 

\begin{multline*}
\P(C_1 \in a|b|c\text{ and }C_1\ra_{S_1}C_2\in ab)\\  \le \P(a|b|c) - \P(C_1 \in a|b|c\text{ and }C_1\ra_{S_1}C_2\in a|b|c) 
= \P(a|b|c) - \frac{\P(a|b|c)^2}{\P(a|c \cup b|c)}
\end{multline*}

and 

\begin{multline*}
\P(C_1 \in a|b|c\text{ and }C_1\ra_{S_2}C_2\in ac)\\ \le \P(a|b|c) - \P(C_1 \in a|b|c\text{ and }C_1\ra_{S_2}C_2\in a|b|c) 
= \P(a|b|c) - \frac{\P(a|b|c)^2}{\P(a|b \cup b|c)}. 
\end{multline*}

Combining these bounds, we get the needed equation \eqref{eq:q2}.

For infinite $G$, the theorem follows by passing to the limit.
\qed

Now we are equipped to prove Theorem~\ref{thm:conj2}.
\subsection{Proof of Theorem~\ref{thm:conj2}}
By interchanging vertices $a$ and $b$ in \eqref{eq:q2}, we get

\begin{equation}\label{eq:q2ab}
\P(a|b|c) + \P(a|b \cup b|c)^2 \ge \frac{\P(a|b|c)^2}{\P(a|b \cup a|c)} + \frac{\P(a|b|c)^2}{\P(a|c \cup b|c)}.
\end{equation}

We rewrite it as

$$\P(a|b|c) + (\P(a|b|c) + \P(ac|b))^2 \ge \frac{\P(a|b|c)^2}{\P(a|b|c) + \P(a|bc)} + \frac{\P(a|b|c)^2}{\P(a|b|c) + \P(ab|c)}.$$

Assume the contrary: let $\varepsilon$ be the counterexample to the conjecture. We will show that small enough $\delta$ contradicts this inequality. Since $\P(ab|c) < \delta$ and $\P(ac|b) < \delta$, we get

$$\P(a|b|c) + (\P(a|b|c) + \delta)^2 \ge \frac{\P(a|b|c)^2}{\P(a|b|c) + \delta} + \frac{\P(a|b|c)^2}{1 - \P(abc)}.$$

Let us move the summands:

$$\P(a|b|c) - \frac{\P(a|b|c)^2}{\P(a|b|c) + \delta}  \ge \frac{\P(a|b|c)^2}{1 - \P(abc)} - (\P(a|b|c) + \delta)^2.$$

Converting to the common denominator:

$$\delta \ge \frac{\delta\P(a|b|c)}{\P(a|b|c) + \delta}  \ge \frac{- 2\delta\P(a|b|c) - \delta^2 + \P(abc)(\P(a|b|c) + \delta)^2}{1 - \P(abc)}.$$

Finally, we multiply both parts by $1 - \P(abc)$ and estimate assuming $\P(abc) \ge \varepsilon$ and $\P(a|b|c)\ge \varepsilon$:

$$4\delta \ge \delta - \delta\P(abc) + 2\delta\P(a|b|c) + \delta^2 \ge \P(abc)(\P(a|b|c) + \delta)^2 \ge \P(abc)\P(a|b|c)^2 \ge \varepsilon^3.$$

Now we see that $\delta < \frac{\varepsilon^3}4$ contradicts Lemma~\ref{lm:q2}, thus proving the conjecture.
\qed

\subsection{Remarks on the result}
We hope that our tools can attack the notorious Conjecture~\ref{conj:3}. For now, we can only say using Theorem~\ref{thm:DV8} that 
$$\P(ab|c) < \delta \implies \P(abc) - 8\P(ac)\P(bc) < \varepsilon.$$
However, we improved the bounds on $\min\big(\P(abc), \P(a|b|c)\big)$. In \cite{GZ} this quantity was called $\alpha_3$ and the upper bound on it was $0.369$. Without loss of generality, $$\P(a|b\cup a|c) = \min(\P(a|b\cup a|c), \P(a|b\cup b|c), \P(a|c\cup b|c)).$$ Then from the inequality \eqref{eq:q2}, we get the upper bound $t$ on $\alpha_3$. Indeed, the optimum is achieved when $\P(a|b\cup a|c)$ and $\P(a|b|c)$ are as large as possible and $\P(a|b\cup b|c), \P(a|c\cup b|c)$ are as small as possible. It leads us to an equation

%\begin{equation}
$$t + \left(t + \frac{1-2t}{3}\right)^2 \ge 2\frac{t^2}{t + \frac{1-2t}{3}}.$$
%\end{equation}

This can be simplified to $$\frac{t^3 - 42 t^2 + 12 t + 1}{t + 1} \ge 0,$$
and the only root of the numerator on the $(0, 1)$ interval is $\approx0.356$. This is better than the previous upper bound on $\alpha_3$, but is still quite apart from the best lower bound of~$0.29065$.

\section{General form of decision tree inequalities}\label{sec:zipper}

Here we consider a generalized setup, where the decision tree generates a configuration as it goes rather than reveals what was hidden. Assume that at each node $N$ of the decision tree $T$, it chooses an edge $e(N)$ of the graph and makes a decision $D(N) \in \{1, 2\}$, and then generates a random element of one of two probability spaces $\Omega_1(e)$ or $\Omega_2(e)$ according to the measures $\mu_1(e)$ or $\mu_2(e)$ respectively. If $N$ is a decision node, it also has links to $|\Omega_i|$ vertices, where $i = D(N)$. Assume that each path from $N_0$ to a leaf node contains all edges once. So, every edge is assigned an element of $\Omega_1(e)$ or $\Omega_2(e)$, and the tree $T$ builds a random configuration $C \in \Omega = \prod_{e\in E}\big(\Omega_1(e) \cup \Omega_2(e)\big)$. We use $\P$ to refer to the induced distribution.

Assume that event $A \subseteq \Omega$ is such that for every edge $e$ and every configuration $C \in \Omega$ the probability of $A$ is bigger if $e$ is resampled from $\mu_2(e)$ rather than if it is resampled from $\mu_1(e)$. 
It turns out to be the common setup for the HK, vdBK and other inequalities.

\begin{mainlemma}\label{thm:zipper}
Let $T$ be a decision tree in the setup above building configuration $C$. % Let $T_1$ be any tree assigning all edges of $G$ to $\Omega_1$ and $T_2$ be any tree assigning all edges of $G$ to $\Omega_2$. They build configurations $C_1$ and $C_2$. 
Let $A$ be an event in $\Omega$. Let $e$ be an arbitrary edge and $C \in \Omega$ an arbitrary configuration. Denote by $X_1 = X_1(C, e)$ the subset of such $x \in \Omega_1$ and by $X_2 = X_2(C, e)$ the subset of such $x \in \Omega_2$ that $C \ra_{E \setminus \{e\}} x \in A$. Assume that for every $e$ and $C$ one has

\begin{equation}\label{eq:zippercond}
\mu_1(X_1) \le \mu_2(X_2).
\end{equation}

%In our notation it means
% \begin{equation}\label{eq:zippercond}
% \P\left(C \ra_{E \setminus \{e\}} C_1 \in {A}~\Big|~C\right) \le \P\left(C \ra_{E \setminus \{e\}} C_2 \in {A}~\Big|~C\right)
% \end{equation}

Then
\begin{equation}\label{eq:zipper}
\P(C_1 \in A) \le \P(C \in A) \le \P(C_2 \in A).
\end{equation}
\end{mainlemma}
\begin{proof}
The proof generalizes the proofs of Theorems \ref{thm:HK} and \ref{thm:vdBK}. We only prove the first inequality of \eqref{eq:zipper}, since the second inequality is proved in the same manner.
We induct on the number of nodes of $T$ with $D(N)=2$. If no such nodes exist, the inequality turns into equality. Otherwise, consider the lowest node $N$ with $D(N)=2$. Let tree $T'$ coincide with $T$ up to the vertex $N$, but the node $N'$ has $D(N')=1$ and so its children should now be indexed by $\Omega_1(e(N))$. We achieve it by copying an arbitrary child of $N$ with its subtree under all children of $N'$, so that after getting to $N'$ the edges that were not yet assigned would be assigned a random element of the corresponding $\Omega_1$.

Now we see that each path in $T$ not passing through $N$ exists in $T'$ as well and has the same probability there. For paths in $T$ passing through $N$ and the paths in $T'$ passing through $N'$ the configuration $C$ is the same except for the edge $e(N)=e(N')$. This happens with probability $\delta(N)$ and, conditional on passing through $N$, the probability of $A$ is 
% $\P\left(C \ra_{E \setminus \{e\}} C_2 \in {A}~\Big|~C\right)$ for $T$ and $\P\left(C \ra_{E \setminus \{e\}} C_1 \in {A}~\Big|~C\right)$ for $T'$. 
$\mu_1(X_1(C, e))$ for $T$ and $\mu_2(X_2(C, e))$ for $T'$.
So, 
% \begin{multline*}
% \P(C(T) \in A) - \P(C(T') \in A) \\= \delta(N)\left(\P\left(C \ra_{E \setminus \{e\}} C_2 \in {A}~\Big|~C\right)-\P\left(C \ra_{E \setminus \{e\}} C_1 \in {A}~\Big|~C\right) \right) \ge 0,
% \end{multline*}
$$\P(C(T) \in A) - \P(C(T') \in A) = \delta(N)\Big(\mu_1\big(X_1(C, e)\big) -\mu_2\big(X_2(C, e)\big)\Big) \ge 0,$$
by condition \eqref{eq:zippercond}. Now, from the induction hypothesis we obtain $\P(C(T_1) \in A) \le \P(C(T') \in A) \le \P(C(T) \in A)$. The second inequality in \eqref{eq:zipper} is proved analogously.
\end{proof}

Theorems \ref{thm:HK} and \ref{thm:vdBK} follow from the Main Lemma~\ref{thm:zipper} with the correct choice of $\mu_1$ and $\mu_2$. In Theorem~\ref{thm:HK} $\Omega_1(e) = \Omega_2(e) = \{00, 01, 10, 11\}$. If $p$ is the probability of $e$ being open in the original percolation, then $\mu_1$ assigns probabilities of $(1-p)^2, p(1-p), p(1-p)$ and $p^2$ to the elements respectively and $\mu_2$ assigns probabilities of $1-p, 0, 0$ and $p$.

In Theorem~\ref{thm:vdBK}, unary nodes generate a random element of $\Omega_1(e) = \{0, 1\}$ and binary nodes generate a random element of $\Omega_2(e) = \{00, 01, 10, 11\}$. We consider the event $A \sq_S B$ on $\Omega$, where $S$ is the set of edges generated by unary nodes. It is easy to check that condition \eqref{eq:zippercond} holds. Thus, Main Lemma~\ref{thm:zipper} proves Theorem~\ref{thm:vdBK}. In fact, the extra generality helps to spot further generalizations.

\begin{theorem}\label{thm:strongBK}
Assume each edge $e\in E$ is assigned some $p(e) \in [0, 1]$. Let $\Omega_1(e)$ be the set ${0, 1, 2}$ and $\mu_1$ assign the probabilities $(1-p)^2$, $2p(1-p)$ and $p^2$ to these outcomes, respectively. Let $\Omega_2$ be the set $\{00, 01, 10, 11\}$ and $\mu_2$ assign the probabilities $(1-p)^2$, $p(1-p)$, $p(1-p)$ and $p^2$. Let $A$ and $B$ be two increasing events on $\{0, 1\}^E$. Denote by $A \bowtie B$ the following event on $\prod_{e \in E}(\Omega_1(e) \cup \Omega_2(e))$:
\begin{equation}\label{eq:bowtie}
\begin{split}
A \bowtie B = \Biggl\{C \in \prod_{e \in E}(\Omega_1(e) \cup \Omega_2(e))\text{ s.t. }\exists w_1, w_2 \subseteq E\text{ s.t. }\mathrm{Ind}[w_1]\in A, \mathrm{Ind}[w_2]\in B, \\ \text{ and }C(e) \in \begin{cases}
\{1, 2, 10, 11\}, \text{ if }e\in w_1,\\
\{1, 2, 01, 11\}, \text{ if }e \in w_2,\\
\{2, 11\}, \text{ if }e \in w_1 \cap w_2.
\end{cases} \Biggl\}
\end{split}
\end{equation}

Let $C_1$, $C$ and $C_2$ be the configurations built by decision trees as above. Then, 
$$\P(C_1 \in A \bowtie B) \le \P(C \in A \bowtie B) \le \P(C_2 \in A \bowtie B).$$
\end{theorem}

Informally, $A \bowtie B$ is the same as $A \sq B$, but there is an extra probability $p^2$ for a unary node to produce a double edge that can be used in both witnesses.

\begin{proof}
We need to prove the condition \eqref{eq:zippercond}. Consider a configuration $C$ on $\prod_{e \in E}(\Omega_1(e) \cup \Omega_2(e))$, defined up to some edge $e$. 
% Now we have the following possibilities:

% \begin{enumerate}
%     \item There are witnesses $w_1$ and $w_2$ of $A$ and $B$ as in the definition \eqref{eq:bowtie}, that don't use $e$.

%     \item There are subsets $w_1, w_2 \subseteq E \setminus \{e\}$, such that $w_1 \cup \{e\}$ is a witness 
% \end{enumerate}

Note that $\Omega_1$ has a natural linear ordering and $\Omega_2$ has a natural partial ordering. These orders agree with the definition of $\bowtie$ in the sense that if $x < y$ and $C \ra_{E \setminus \{e\}} x \in A \bowtie B$, then $C \ra_{E \setminus \{e\}} y \in A \bowtie B$. 
Let $X_1$ be the subset of such $x \in \Omega_1$ and $X_2$ be the subset of such $x \in \Omega_2$ that $C \ra_{E \setminus \{e\}} x \in A \bowtie B$. Then both $X_1$ and $X_2$ are closed upward. The theorem statement is then equivalent to $\mu_1(X_1) \le \mu_2(X_2)$.

So now there are four possibilities for $X_1$. It can be either $\varnothing$, $\{2\}$, $\{1, 2\}$, or $\{0, 1, 2\}$. We analyze these cases separately.

\begin{enumerate}
\item $X_1 = \varnothing$: then, obviously, $0 = \mu(X_1) \le \P(X_2)$.

\item $X_1 = \{2\}$: then $11 \in X_2$. Indeed, consider the witnesses $w_1$, $w_2$ for $C \ra_{E \setminus \{e\}} 2 \in A \bowtie B$. Same $w_1$, $w_2$ would witness $C \ra_{E \setminus \{e\}} 2 \in A \bowtie B$, because the definition \eqref{eq:bowtie} does not distinguish between $2$ and $11$. So, the probability of $X_2$ is at least $\mu_2(11) = p^2 = \mu_1(2)$.

\item $X_1 = \{1, 2\}$: then $01 \in X_2$ or $10 \in X_2$. Indeed, let $w_1$, $w_2$ be the witnesses for $C \ra_{E \setminus \{e\}} 1 \in A \bowtie B$. At most one of them should contain $e$ and they can not both not contain it, because otherwise they would be witnesses for $C \ra_{E \setminus \{e\}} 0 \in A \bowtie B$ as well. Without loss of generality, $e \in w_1$. Then $w_1$ and $w_2$ are a witness for $C \ra_{E \setminus \{e\}} 10 \in A \bowtie B$ and $10 \in X_2$ as well as $11$. So, $\mu_2(X_2) \ge \mu_2(10)+\mu_2(11) = p = \mu_1(X_1)$. Note that this case contributes to the inequality if $X_2 = \{10, 01, 11\}$, since in other cases we can actually prove $\mu_1(X_1) = \mu_2(X_2)$.

\item $X_1 = \{0, 1, 2\}$: then $00 \in X_2$. Indeed, let $w_1$, $w_2$ be the witnesses for $C \ra_{E \setminus \{e\}} 0 \in A \bowtie B$. Both of them avoid $e$, so they witness $C \ra_{E \setminus \{e\}} 00 \in A \bowtie B$ as well. So $X_2 = \Omega_2$ and $\mu_1(X_1) = 1 = \mu_2(X_2)$.
\end{enumerate}    
\end{proof}

Along with the vdBK inequality, the paper \cite{BK} shows the following stronger result:

\begin{theorem}[{\cite[eq. (3.6)]{BK}}]\label{thm:strongBK2}
Let $A_1$, \dots, $A_n$ and $B_1$, \dots, $B_n$ be increasing events on $\{0, 1\}^E$.

Then

$$\P(A_1 \sq B_1 \cup \dots \cup A_n \sq B_n) \le \P(A_1 \times B_1 \cup \dots \cup A_n \times B_n),$$
where the second event is a subset of $\{00, 01, 10, 11\}^E$ with the probability measure as in Theorem~\ref{thm:strongBK}.
\end{theorem}

We note that the proof of Theorem~\ref{thm:strongBK} also proves the similar statement:

\begin{theorem}[cf. Main Lemma~\ref{thm:zipper}]
In the conditions of Theorem~\ref{thm:strongBK}, for every increasing event $A_1$, \dots, $A_n$ and $B_1$, \dots, $B_n$ on $\{0, 1\}^E$ one has 
\begin{multline*}
\P(C_1 \in A_1 \bowtie B_1 \cup \dots \cup A_n \bowtie B_n) \\\le \P(C \in A_1 \bowtie B_1 \cup \dots \cup A_n \bowtie B_n) \\\le \P(C_2 \in A_1 \bowtie B_1 \cup \dots \cup A_n \bowtie B_n).
\end{multline*}
\end{theorem}

Now we continue applications of our method with the decision tree version of the main theorem from \cite{GP}. In notation from this paper, $f: E\to \{a,b,c,d\}$ is a uniform random coloring of the edges of~$G$, where each edge is colored uniformly and independently. Denote by
$E_s$, $s \in \{a,b,c,d\}$, a subset of edges of the corresponding color.
Similarly, for every two distinct colors $s, t \in \{a,b,c,d\}$, let
$E_{st}:=E_s\cup E_t$.
One can think of $E_{st}$ as either a $\frac12$-percolation or a uniformly random
subset of edges of~$G$, so that $G_{st}=(V, E_{st})$ is a uniform random subgraph of~$G$.

\begin{theorem}[{\cite[first part of Theorem 1]{GP}}]\label{t:events}
Let $\cU, \cV, \cW$ be closed upward graph properties.  Denote by
$\cU_{ab}$, $\cV_{ac}$  and $\cW_{bc}$  the corresponding
properties of $G_{ab}$, $G_{ac}$ and $G_{bc}$, respectively.  Then the events $\cU_{ab}$, $\cV_{ac}$ and $\cW_{bc}$ are pairwise independent, but have negative mutual dependence:
\begin{equation}\label{eq:GP}
\bP(\cU_{ab} \cap \cV_{ac} \cap \cW_{bc}) \le \bP(\cU_{ab}) \bP(\cV_{ac}) \bP(\cW_{bc}),
\end{equation}
where the probability is over uniform random colorings $f: E\to \{a,b,c,d\}$.  
\end{theorem}

From Main Lemma \ref{thm:zipper} we get the decision tree version.
Let $\Omega_1$ be the set of triplets $\{000, 011, 101, 110\}$ and $\mu_1$ be the measure assigning $\frac{1}{8}$ to each element of $\Omega_2$. Similarly, let $\Omega_2$ be the full set of triplets $$\{000, 001, 010, 011, 100, 101, 110, 111\}$$ and $\mu_2$ be the measure that assigns $\frac{1}{4}$ to each element of $\Omega_2$. 
We say that an element $C$ of $\prod_{e \in E}(\Omega_1(e) \cup \Omega_2(e))$ belongs to $\cU\times \cV \times \cW$ if the configuration $E_1 \in \{0, 1\}^E$ formed by the first digits on the edges belongs to $\cU$, the configuration $E_2$ formed by the second digits on the edges belongs to $\cV$ and the configuration $E_3$ formed by the third digits on the edges belongs to $\cW$. 

\begin{theorem}\label{thm:colored}
Let $C_1$, $C$ and $C_2$ be the configurations built by decision trees as above. Then, 
\begin{equation}\label{eq:gp1}
    \P(C_1 \in \cU\times \cV) = \P(C \in \cU\times \cV) = \P(C_2 \in \cU\times \cV).
\end{equation}
and
\begin{equation}\label{eq:gp2}
    \P(C_1 \in \cU\times \cV \times \cW) \le \P(C \in \cU\times \cV \times \cW) \le \P(C_2 \in \cU\times \cV \times \cW).
\end{equation}
\end{theorem}
\begin{proof}
The ``pairwise independent'' part of Theorem~\ref{t:events} is an easy part. It also generalizes to the decision tree setup as \eqref{eq:gp1}, since, without the third coordinate, the first two coordinates are uniformly distributed in both $\mu_1$ and $\mu_2$.

If $C \in \prod_{e \in E}\Omega_1(e)$, then by the choice of $\Omega_1$, we see that $E_3 = E_1 \oplus E_2$ is an edgewise exclusive or of independent $E_1$ and $E_2$, just like $E_{bc} = E_{ab} \oplus E_{ac}$. So, the probability on the left coincides with $\bP(\cU_{ab} \cap \cV_{ac} \cap \cW_{bc})$. The event on the right is an intersection of three independent events depending on $E_1$, $E_2$ and $E_3$, so its probability coincides with $\bP(\cU_{ab}) \bP(\cV_{ac}) \bP(\cW_{bc})$.

Let $X_1$ be the subset of such $x \in \Omega_1$ and $X_2$ be the subset of such $x \in \Omega_2$ that $C \ra_{E \setminus \{e\}} x \in \cU\times \cV \times \cW$. Note that $X_1 = X_2 \cap \Omega_2$, so we only need to analyze the possibilities for $X_2$. Since the condition for $x \in \cU\times \cV \times \cW$ splits into $3$ independent conditions for $3$ coordinates, $X_2$ is a Cartesian product of $3$ sets. Moreover, since $\cU$, $\cV$ and $cW$ are increasing, $X_2$ is also increasing. This leaves us with a few options, up to the coordinate permutation.

\begin{enumerate}
\item $X_2 = \varnothing$: then $X_1 = \varnothing$ and $\mu_1(X_1) = 0 = \mu_2(X_2)$.

\item $X_2 = \{111\}$: then $X_1 = \varnothing$ and $\mu_1(X_1) = 0 < \frac{1}{8} = \mu_2(X_2)$. Note that this is the only case where the inequality is strict.

\item $X_2 = \{111, 110\}$: then $X_1 = \{110\}$ and $\mu_1(X_1) = \frac{1}{4} = \mu_2(X_2)$.

\item $X_2 = \{111, 110, 101, 100\}$: then $X_1 = \{110, 101\}$ and $\mu_1(X_1) = \frac{1}{2} = \mu_2(X_2)$.

\item $X_2 = \Omega_2$: then $X_1 = \Omega_1$ and $\mu_1(X_1) = 1 = \mu_2(X_2)$.
\end{enumerate}

By Main Lemma \ref{thm:zipper}, we are done.
\end{proof}

The last application is somewhat similar. Let $$\Omega_1 = \Omega_2 = \{000, 001, 010, 011, 100, 101, 110, 111\}.$$
Let $\mu_1$ be the mixture of the uniform distribution $\mu_{11}$ on $\{000, 111\}$ with the coefficient $\frac{2}{3}$ and the uniform distribution $\mu_{12}$ on $\Omega_1$ with the coefficient $\frac{1}{3}$. 
Let $\mu_2$ be the mixture of the uniform distribution $\mu_{21}$ on $\{000, 011, 100, 111\}$ with the coefficient $\frac{1}{3}$, the uniform distribution $\mu_{22}$ on $\{000, 010, 101, 111\}$ with the coefficient $\frac{1}{3}$, and the uniform distribution $\mu_{23}$ on $\{000, 001, 110, 111\}$ with the coefficient $\frac{1}{3}$.
\begin{theorem}
Let $C_1$, $C$ and $C_2$ be the configurations built by decision trees as above. Then, 
\begin{equation}\label{eq:richards}
    \P(C_1 \in \cU \times \cV \times \cW) \ge \P(C \in \cU \times \cV \times \cW) \ge \P(C_2 \in \cU \times \cV \times \cW).
\end{equation}
\end{theorem}
\begin{proof}
Let $X$ be the subset of such $x \in \Omega_1 = \Omega_2$ that $C \ra_{E \setminus \{e\}} x \in \cU\times \cV \times \cW$. As in the previous proof, $X$ can only be an increasing product of three events. By Main Lemma~\ref{thm:zipper}, we are left to check that $\mu_1(X) \ge \mu_2(X)$. So, without loss of generality, we have the following cases.

\begin{enumerate}
\item $X = \varnothing$: then $\mu_1(X) = 0 = \mu_2(X)$.

\item $X =  \{111\}$: then $\mu_1(X) = \frac{9}{24} > \frac{1}{4} = \mu_2(X)$. Note that this is one of the two cases where the inequality is strict.

\item $X = \{111, 110\}$: then $\mu_1(X_1) = \frac{10}{24} > \frac{4}{12} = \mu_2(X_2)$. Note that this is one of the two cases where the inequality is strict.

\item $X = \{111, 110, 101, 100\}$: then $\mu_1(X_1) = \frac{1}{2} = \mu_2(X_2)$.

\item $X = \Omega_2$: then $\mu_1(X_1) = 1 = \mu_2(X_2)$.
\end{enumerate}

\end{proof}

\begin{remark}
This final application of Main Lemma \ref{thm:zipper} stems from the work of Richards \cite{R}. His paper provides an incorrect proof for the inequality 
\begin{equation}\label{eq:sahi}
2\P(\cU\cap\cV\cap\cW) + \P(\cU)\P(\cV)\P(\cW) \ge \P(\cU)\P(\cV\cap\cW) + \P(\cV)\P(\cU\cap\cW) + \P(\cW)\P(\cV\cap\cU)
\end{equation}

The proof mimics the proof of the HK inequality and utilizes induction. The induction step implicitly worked in the space of triples of configurations and effectively was equivalent to equation \eqref{eq:richards}. % ... %TODO: write about Richards paper. 
Inequality \eqref{eq:sahi} is still a conjecture. Sahi \cite{Sa} generalized this inequality to a series of conjectured inequalities. There are partial results in the direction of these conjectures \cite{LS}.
\end{remark}

\section{Inequalities for disjoint paths between two vertices}\label{sec:ab}
\subsection{Proof of Theorem~\ref{thm:ab23}}
Finally, after studying the connectivity events for $3$ vertices, we study the minimal case -- inequalities concerning connections for just two points. Although it may seem that there is not enough variation -- $a$ and $b$ can be either connected or disconnected, we study the events of the form $ab^{\sq n} := ab \sq ab \sq \dots \sq ab$ ($n$ times). Note that in general $\sq$ is not associative, but this particular event means that there are $n$ nonintersecting paths from $a$ to $b$ passing through open edges. Thus, this definition does not depend on the order of operations.
When $b$ is a ghost vertex, $\P(ab^{\sq n})$ is related to the monochromatic arms exponents.

% We prove a partial result in the direction of Conjecture~\ref{conj:3}. 

\begin{proof}[Proof of {Theorem~\ref{thm:ab23}}]
Let $G$ be finite.
Without loss of generality, the face to which $a$ and $b$ both belong is an outer face. This allows us to run a right-hand rule walk on it and to talk about the ``right'' and ``left'' side of every path.
Let $T$ be a decision tree that runs a right-hand rule walk starting from $a$, until it runs into $b$, and put its edges in $S$. If the walk reaches $b$, then part of the edges in this walk form the path $P_1$ that is the rightmost path from $a$ to $b$. It means that for every path $P$ from $a$ to $b$, all vertices of $P_1$ lie on $P$ or to the right of it. In this case, run the second right-hand rule path from $a$, not taking the edges already considered. If this walk also reaches $b$, then part of the edges in the walk should form the path $P_2$ which is the second rightmost path from $a$ to $b$. It means that for all paths $P$ and $Q$ that don't share edges and $Q$ lies to the right of $P$, the path $P_2$ lies to the right of $P$. 

Now $T$ is a decision tree for the event $ab^{\sq 2}$. If this event occurs, then we can continue $T$ to the tree $T'$ that runs the right-hand rule walk from $a$ once again. Then $T'$ is a decision tree for the event $ab^{\sq 3}$. Now, from Theorem~\ref{thm:CS} we get

\begin{equation}\label{eq:arms1}
\P(C_1 \in ab^{\sq 3}, C_1 \ra_S C_2 \in ab^{\sq 3}) \ge \frac{\P(ab^{\sq 3})^2}{\P(ab^{\sq 2})}.
\end{equation}

Also, from Theorem~\ref{thm:vdBK}, we get the other estimate. Indeed, if $C_1 \in ab^{\sq 3}$ and $C_1 \ra_S C_2 \in ab^{\sq 3}$, then there are paths $P_3$ in $C_1|_{\overline{S}}$ that completes the triple of nonintersecting paths $P_1$, $P_2$ and $P_3$ in $C_1$ and $P'_3$ in $C_2|_{\overline{S}}$ that completes the triple of nonintersecting paths $P_1$, $P_2$ and $P'_3$ in $C_1 \ra_S C_2$. So we have a pair of witnesses $(P_1 \cup P_3, P_2 \cup P'_3)$ for the event $ab^{\sq 2} \sq_{S} ab^{\sq 2}$. Now by Theorem~\ref{thm:vdBK} we get

\begin{equation}\label{eq:arms2}
\P(C_1 \in ab^{\sq 3}, C_1 \ra_S C_2 \in ab^{\sq 3}) \le \P(ab^{\sq 2} \sq_{S} ab^{\sq 2}) \le \P(ab^{\sq 2})^2.
\end{equation}
Combining equations \eqref{eq:arms1} and \eqref{eq:arms2}, we get the needed \eqref{eq:arms}.

For the infinite $G$, the result follows by standard limit arguments.

\end{proof}

\subsection{Generalization of Theorem~\ref{thm:ab23}}

\begin{theorem}\label{thm:abklm}
Let $G$ be planar. Assume $a$ and $b$ belong to the same face, $n$ is a natural number and $k, l, m \le n$ are such that $k+l+m=2n$. Then
\begin{equation}\label{eq:manyarms}
\P(ab^{\sq n})^2 \le \P(ab^{\sq k})\P(ab^{\sq l})\P(ab^{\sq m})
\end{equation}
\end{theorem}
\begin{proof}
The proof is analogous to the previous one. Let $T$ be a decision tree that runs $k$ right-hand rule walks from $a$ and puts the edges it meets in $S$. Then $T$ is a decision tree for $ab^{\sq k}$. By Theorem~\ref{thm:CS}, we get
\begin{equation}\label{eq:manyarms1}
\P(C_1 \in ab^{\sq n}, C_1 \ra_S C_2 \in ab^{\sq n}) \ge \frac{\P(ab^{\sq n})^2}{\P(ab^{\sq k})}.
\end{equation}

On the other hand, if $C_1 \in ab^{\sq n}$ and $C_1 \ra_S C_2 \in ab^{\sq n}$, then there are $n-k$ nonintersecting paths from $a$ to $b$ in $C_1|\overline{S}$ and other $n-k$ nonintersecting paths from $a$ to $b$ in $C_2|\overline{S}$. We add them to witnesses $w_1$, $w_2$ of $ab^{\sq l} \sq_S ab^{\sq m}$. Now we split the $k$ paths from $S$ into $n-m$ and $n-l$ paths (we can do it since $n-m+n-l=k$) and add these paths to $w_1$ and $w_2$, respectively. Now this construction gives an estimate

\begin{equation}\label{eq:manyarms2}
\P(C_1 \in ab^{\sq n}, C_1 \ra_S C_2 \in ab^{\sq n}) \le \P(ab^{\sq l} \sq_{S} ab^{\sq m}) \le \P(ab^{\sq l})\P(ab^{\sq m}).
\end{equation}

Combining equations \eqref{eq:manyarms1} and \eqref{eq:manyarms2}, we get the needed \eqref{eq:manyarms}.
\end{proof}

\section{Open problems}

Section~\ref{sec:conj} leaves some open questions. Despite Theorem~\ref{thm:conj2}, the more precise question remains open:

\begin{conjecture}[]\label{conj:3}
For $\varepsilon > 0$, there exists $\delta > 0$, such that 
$$\P(ab|c) < \delta \implies \Big(\P(abc)\P(a|b|c) - \P(ac|b)\P(a|bc)<\varepsilon\Big).$$
\end{conjecture}

Numerical simulations confirm this conjecture, which is as natural as could be.

We also propose a strengthening of the Conjecture~\ref{conj:log} on the probabilities of $ab^{\sq k}$. Consider the example where $G$ consists just of the vertices $a$ and $b$ connected via $N$ edges (or disjoint paths, to keep $G$ simple), each having a probability of $\frac{\lambda}{N}$. Then as $N \to \infty$, the distribution of the number of paths between $a$ and $b$ tends to the Poisson distribution with parameter $\lambda$, so we have $\P(ab^{\sq k}) \to \sum_{i=k}^\infty \frac{\lambda^i}{i!e^\lambda}$. 

\begin{conjecture}\label{conj:pois}
For a given graph $G$ we define the implied $\lambda_k$ as the unique number such that 
$$\P(ab^{\sq k}) = \sum_{i=k}^\infty \frac{\lambda_k^i}{i!e^{\lambda_k}}.$$
We conjecture that $\{\lambda_k\}$ is a decreasing sequence.
\end{conjecture}

\section{Acknowledgements}

The author thanks Igor Pak for posing the problem and his useful comments, Aleksandr Zimin and Dmitry Krachun for many fruitful discussions, Yu Feng for explaining the Delfino--Viti formula and Tom Hutchcroft and Gady Kozma for their careful reading of the draft of the paper.

\bibliographystyle{plain}

\end{document}